\documentclass[12pt]{amsart}
\usepackage{amsmath, amssymb, amsthm, bm, float, mathtools, enumitem, caption}
\usepackage{hyperref}

\usepackage[noabbrev,capitalize]{cleveref}
\usepackage{adjustbox}
\crefname{equation}{}{}
\usepackage{ytableau}
\usepackage{graphicx, tikz}
\usepackage[textheight=8.75in, textwidth=6.75in]{geometry}
\usepackage{subcaption}
\usepackage{microtype}

\newtheorem{theorem}{Theorem}[section]

\newtheorem{corollary}[theorem]{Corollary}

\newtheorem*{conjecture*}{Conjecture}

\theoremstyle{definition}

\theoremstyle{remark}
\newtheorem*{remark}{Remark}
\newtheorem*{example}{Example}

\numberwithin{equation}{section}

\newcommand{\N}{\mathbb N}

\newcommand{\fm}{\mathfrak m}
\newcommand{\fmodd}{\mathfrak{m}_{{\text {\rm odd}}}}

\newcommand{\SL}{\mathrm{SL}}

\newcommand{\Z}{\mathbb Z}

\title[Remarks on MacMahon $q$-series]{Remarks on MacMahon's $q$-series}
\date{\today}
\author{Ken Ono \and Ajit Singh}

\address{Dept. of Mathematics, University of Virginia, Charlottesville, VA 22904}
\email{ko5wk@virginia.edu}
\email{ajit18@iitg.ac.in}

\begin{document}
\begin{abstract} In his important 1920 paper on partitions, MacMahon defined the partition generating functions
\begin{displaymath}
\begin{split}
A_k(q)=\sum_{n=1}^{\infty} \fm(k;n)q^n&:=\sum_{0< s_1<s_2<\cdots<s_k} \frac{q^{s_1+s_2+\cdots+s_k}}{(1-q^{s_1})^2(1-q^{s_2})^2\cdots(1-q^{s_k})^2},\\
C_k(q)=\sum_{n=1}^{\infty} \fmodd(k;n)q^n&:=\sum_{0< s_1<s_2<\cdots<s_k} \frac{q^{2s_1+2s_2+\cdots+2s_k-k}}{(1-q^{2s_1-1})^2(1-q^{2s_2-1})^2\cdots(1-q^{2s_k-1})^2}.
\end{split}
\end{displaymath} 
These series give infinitely many formulas for two prominent generating functions.
For each non-negative $k$, we prove that  $A_k(q), A_{k+1}(q), A_{k+2}(q),\dots$ (resp. $C_k(q), C_{k+1}(q), C_{k+2}(q),\dots$) give the generating function for the 3-colored partition function $p_3(n)$  (resp.  the overpartition function $\overline{p}(n)$). 
\end{abstract}

\maketitle
\section{Introduction and Statement of Results}

In an important paper on integer partitions, MacMahon \cite{MacMahon} introduced the family of $q$-series 
\begin{align} \label{Mac1} 
A_k(q):=\sum_{0< s_1<s_2<\cdots<s_k} \frac{q^{s_1+s_2+\cdots+s_k}}{(1-q^{s_1})^2(1-q^{s_2})^2\cdots(1-q^{s_k})^2}.
\end{align}
For positive integers $k,$ we have that
\begin{equation}\label{mk}
A_k(q)=\sum_{n=1}^{\infty}\fm(k;n)q^n=\sum_{\substack{0<s_1<s_2<\dots<s_k\\  (m_1,\dots,m_k)\in \N^k}} m_1m_2\dots m_k q^{m_1s_1+m_2s_2+\dots +m_ks_k},
\end{equation}
and so $A_k(q)$ is a natural partition generating function. Indeed, we have that $\fm(k;n)$ is the sum of the products of the {\it part multiplicities}  for partitions of $n$ with $k$ distinct part sizes. 

These series connect partitions to disparate areas of mathematics. In elementary number theory,
MacMahon realized some of the $A_k(q)$ as generating functions for divisor sums.
For example, he found that
\begin{displaymath}
\begin{split}
A_1(q)&=\sum_{n=1}^{\infty} \sigma_1(n)q^n=1+3q^2+4q^3+7q^4+6q^5+\dots,\\
A_2(q)&=\frac{1}{8}\sum_{n=1}^{\infty} \left( (-2n+1)\sigma_1(n)+\sigma_3(n) \right)=q^3+3q^4+9q^5+\dots,
\end{split}
\end{displaymath}
where $\sigma_{\nu}(n):=\sum_{d\mid n}d^{\nu}.$
Extending beyond number theory, these series arise in the study of  Hilbert schemes,  $q$-multiple zeta-values, representation theory, and  topological string theory (for example, see \cite{R1, R2, R3, R4}).
Recent research has focused on the quasimodularity of the $A_k(q)$ 
(for example, see \cite{AOS, AAT, Andrews-Rose, Bachmann, Bringmann, Rose}). Andrews and Rose
\cite{Andrews-Rose, Rose} proved that $A_k(q)$ is a linear combination of quasimodular forms on $\SL_2(\Z)$ with weights $\leq 2k.$

 In this note, we instead focus on the combinatorial properties of MacMahon's series. We show that they satisfy infinitely many systematic identities, that, in turn, illustrate the ubiquity of the $\fm(k;n)$ partition functions.
To set the stage, we offer some terms of $A_0(q), A_1(q),\dots, A_5(q):$
\begin{displaymath}
\begin{split}
A_0(q)&:=q^0,\\
A_1(q)&=q+3q^2+4q^3+7q^4+6q^5+\dots, \\
A_2(q)&=q^3+3q^4+9q^5+15q^6+30q^7+\dots, \\
A_3(q)&=q^6+3q^7+9q^8+22q^9+42q^{10}+\dots, \\
A_4(q)&=q^{10}+3q^{11}+9q^{12}+22q^{13}+51q^{14}+\dots, \\
A_5(q)&=q^{15} + 3q^{16} + 9q^{17} + 22q^{18} + 51q^{19}  + \dots
\end{split}
\end{displaymath}
As these examples suggest, the $A_k(q)$ behave well as $k\rightarrow +\infty.$ Indeed, in terms of the $q$-Pochhammer symbol
$$
(a;q)_{\infty}:=(1-a)(1-aq)(1-aq^2)\dots,
$$
and Jacobi's famous identity for $(q;q)_{\infty}^3,$ 
it is known  that (see Theorem 1.1 of \cite{AOS})
\begin{equation}
\label{Limit_k}
\frac{1}{(q;q)_{\infty}^3}=q^{-\frac{k^2+k}{2}} A_k(q)+ O(q^{k+1}).
\end{equation}
In a recent preprint, Bringmann, Craig, van Ittersum and Pandey \cite{Bringmann} obtain further such results relating infinite products with MacMahon-type $q$-series.

It is natural to ask whether (\ref{Limit_k}) is a glimpse of explicit identities, one for each non-negative integer $k$. We show that this is indeed the case, where $q^{-\frac{k^2+k}{2}}A_k(q)$ is simply the first summand of a closed formula involving $A_k(q), A_{k+1}(q), A_{k+2}(q),\dots.$ 

\begin{theorem}\label{Ak}
	If $k$ is a non-negative integer, then we have
	$$\frac{1}{(q;q)_{\infty}^3}=
	q^{-\frac{k^2+k}{2}}\sum_{m=k}^\infty \binom{2m+1}{m+k+1}A_{m}(q).
	$$
\end{theorem}

To further appreciate these identities, we note, for positive $m$,  that (\ref{Limit_k}) implies
\begin{equation}\label{AkShift}
q^{-\frac{k^2+k}{2}}A_{k+m}(q)=q^{\frac{m(m+1)}{2}+mk}+\dots.
\end{equation}
The first terms of these $q$-series have exponents that grow quadratically $m.$   Therefore, we can use Theorem~\ref{Ak} to compute initial segments of 
\begin{align}\label{GF-3-color}
\frac{1}{(q;q)_{\infty}^3}=\sum_{n\geq 0} p_3(n)q^n,
\end{align}
the generating function for the 3-colored partition function $p_3(n),$  using a ``small number'' of summands.  In this way, we obtain a doubly infinite family of formulas relating the 3-colored partition function to MacMahon's $\fm(k;n)$ partition functions.

\begin{corollary}\label{Ak_Approximation}
If $k$ and $j$ are non-negative integers, then
	$$\frac{1}{(q;q)_{\infty}^3}=q^{-\frac{k^2+k}{2}}\sum_{m=0}^{j} \binom{2m+2k+1}{m+2k+1}A_{m+k}(q) + O\left(q^{\frac{(j+1)(j+2k+2)}{2}}\right).
	$$	
In particular, if $n<(j+1)(j+2k+2)/2,$ then we have
$$p_3(n)= \sum_{m=0}^j  \binom{2m+2k+1}{m+2k+1}\fm\left (m+k;n+\frac{k^2+k}{2}\right).
$$
\end{corollary}

\begin{remark}
Letting $j=1$ in Corollary~\ref{Ak_Approximation} gives Theorem~1.1 (ii) of \cite{AOS}.
\end{remark}

\begin{example} If $k=100$ and $j=2,$ then Corollary~\ref{Ak_Approximation} gives
\begin{displaymath}
\frac{1}{(q;q)_{\infty}^3}=q^{-5050}\cdot (A_{100}(q)+203A_{101}(q)+20910A_{102}(q)) + O(q^{306}).
\end{displaymath}
Therefore, for $n<306,$ we have
$$
p_3(n)=\fm(100; n+5050)+203\fm(101; n+5050)+20910\fm(102;n+5050).
$$
\end{example}

In addition to the $A_k(q),$ MacMahon also introduced \cite{MacMahon}  the $q$-series
\begin{align} \label{Mac2} 
C_k(q)=\sum_{n=1}^{\infty} \fmodd(k;n)q^n:=\sum_{0< s_1<s_2<\cdots<s_k} \frac{q^{2s_1+2s_2+\cdots+2s_k-k}}{(1-q^{2s_1-1})^2(1-q^{2s_2-1})^2\cdots(1-q^{2s_k-1})^2}.
\end{align}
The numbers $\fmodd(k;n)$ have the same partition theoretic description as the $\fm(k;n)$, where here the parts are required to be  odd. 
Furthermore,
in analogy with the work of Andrews and Rose \cite{Andrews-Rose, Rose},  Bachmann \cite{Bachmann} proved that each $C_k(q)$ is a finite linear combination of quasimodular forms on $\Gamma_0(2)$ with weight $\leq 2k.$

Here we show that the $C_k(q)$ also enjoy properties that are analogous to those of $A_k(q)$ described above. 
Namely, we prove the following theorem, where $C_0(q):=1.$

\begin{theorem}\label{Ck}
	The following are true.
		 
	 \noindent
	 (1) If $k$ is a non-negative integer, then we have
	 $$
	 q^{-k^2}C_k(q)=\frac{1}{(q^2;q^2)_{\infty}(q;q^2)_\infty^2}+O(q^{2k+1}).
	 $$

	\noindent
	(2) If $k$ is a non-negative integer, then we have
	   $$\frac{1}{(q^2;q^2)_{\infty}(q;q^2)_\infty^2}=
	   q^{-k^2}\sum_{m=k}^\infty \binom{2m}{m+k}C_{m}(q).
	   	   $$
\end{theorem}

Theorems~\ref{Ak} and ~\ref{Ck} establish infinitely many formulas, one for each  integer $k$, between MacMahon's two families of $q$-series  and the reciprocals of the theta functions
\begin{displaymath}
\begin{split}
(q;q)_{\infty}^3&=\sum_{n=0}^{\infty}(-1)^n (2n+1)q^{\frac{n^2+n}{2}}=1-3q+5q^3-7q^6+9q^{10}-\dots,\\
(q^2;q^2)_{\infty}(q;q^2)_\infty^2&=\sum_{n\in \Z} (-1)^n q^{n^2}=1-2q+2q^4-2q^9+2q^{16}-2q^{25}+\dots.
\end{split}
\end{displaymath}
Corollary \ref{Ak_Approximation}, which relates the 3-colored partition function to MacMahon's $\fm(k;n)$ partition functions, relies on the fact that $1/(q;q)_{\infty}^3$ is the generating function of $p_3(n)$. Rather nicely, it turns out that
\begin{equation}\label{GenFcnOver}
\frac{1}{(q^2;q^2)_{\infty}(q;q^2)_\infty^2}=\sum_{n=0}^{\infty} \overline{p}(n)q^n=1+2q+4q^2+8q^3+14q^4+24q^5+\dots,
\end{equation}
where $\overline{p}(n)$ denotes the number of {\it overpartitions} of size $n.$ Recall that an overpartition of $n$ is an ordered sequence of nonincreasing positive integers, where the first occurrence of each integer may be overlined \cite{CorteelLovejoy}. Overpartitions have been the focus of intense research in recent years (for example, see \cite{ADSY, KLO, CorteelLovejoy, Hirschhorn-Sellers, Lovejoy-Osburn, Zhang}).
 Therefore, in analogy with Corollary~\ref{Ak_Approximation}, we obtain the following result.

\begin{corollary}\label{Ck_Approximation}
If $k$ and $j$ are non-negative integers, then
	$$\frac{1}{(q^2;q^2)_\infty(q;q^2)_\infty^2}=q^{-k^2}\sum_{m=0}^{j} \binom{2m+2k}{m+2k}C_{m+k}(q) + O\left(q^{(j+1)(j+2k+1)}\right).
	$$	
In particular, if $n<(j+1)(j+2k+1),$ then we have
$$\overline{p}(n)= \sum_{m=0}^j  \binom{2m+2k}{m+2k}\fmodd\left (m+k;n+k^2\right).
$$
\end{corollary}

\begin{example} If $k=100$ and $j=2,$ then Corollary~\ref{Ck_Approximation} gives
\begin{displaymath}
\frac{1}{(q^2;q^2)_\infty(q;q^2)_\infty^2}=q^{-10000}\cdot (C_{100}(q)+202C_{101}(q)+20706C_{102}(q)) + O(q^{609}).
\end{displaymath}
Therefore, for $n<{609},$ we have
$$
\overline{p}(n)=\fmodd(100;n+10000)+202\fmodd(101;n+10000)+20706\fmodd(102;n+10000).
$$
\end{example}

The proofs of our  results are rather straightforward, and follow from the Jacobi triple product identity. Namely,  we recognize the role of MacMahon's $q$-series as coefficients of power series in  $(z+z^{-1})^{2}$ obtained from this well-known bivariate infinite product. 

\begin{remark} The proofs of Theorems~\ref{Ak} and \ref{Ck} follow along similar lines. They differ in their choice of specialization (i.e. \hspace{-1.3em} changes of variable) of the Jacobi triple product identity. It would be interesting to see if other natural partition generating functions emerge from further specializations, to supplement these results on $p_3(n)$ and $\overline{p}(n).$ Finally, we point out that it would be interesting to carry out a similar analysis for the quintuple and septuple infinite product identities.
\end{remark}

\section*{Acknowledgements}
\noindent
The first author thanks the Thomas Jefferson Fund and the NSF
(DMS-2002265 and DMS-2055118). The second author thanks the support of a Fulbright Nehru Postdoctoral Fellowship. 

\section{Proofs}

\subsection{MacMahon's $A_k(q)$}
 Here we prove Theorem~\ref{Ak} and Corollary~\ref{Ak_Approximation}. 
 
 \begin{proof}[Proof of Theorem~\ref{Ak}]
 We recall the Jacobi triple product identity (see Theorem~2.8 of \cite{Andrews})
 \begin{equation}\label{JTP}
 \sum_{n=-\infty}^{\infty} q^{n^2}z^{n}=\prod_{n=0}^{\infty}(1-q^{2n+2})(1+z^{-1}q^{2n+1})(1+zq^{2n+1}).
 \end{equation}
 By factoring out $(q^2;q^2)_{\infty}$, and letting $z\rightarrow qz^2$, and then letting $q\rightarrow \sqrt{q}$, a simple reindex gives
 \begin{displaymath}
 \begin{split}
 \sum_{n=-\infty}^{\infty}q^{\frac{n(n+1)}{2}}z^{2n}=\left(1+z^{-2}\right)(q;q)_{\infty}\prod_{n=1}^{\infty}(1+z^{-2}q^{n})(1+z^2q^{n}).
 \end{split}
 \end{displaymath}
After straightforward algebraic manipulation, we find
 \begin{align*}
  \sum_{n=-\infty}^{\infty}q^{\frac{n(n+1)}{2}}z^{2n}
  =\left(1+z^{-2}\right)(q;q)_{\infty}\prod_{n=1}^{\infty}\left((1-q^{n})^2+(z+z^{-1})^2q^{n}\right).
  \end{align*}
  After factoring out $(q;q)_{\infty}^2$ from the infinite product,  we obtain
$$  
    \sum_{n=-\infty}^{\infty}q^{\frac{n(n+1)}{2}}z^{2n}=\left(1+z^{-2}\right)(q;q)_{\infty}^3\prod_{n=1}^{\infty}\left(1+\frac{q^{n}}{(1-q^{n})^2}\cdot (z+z^{-1})^2\right).
 $$
 Thanks to definition (\ref{Mac1}), we find that the infinite product on the right, as a power series in $(z+z^{-1})^2,$ is the generating function for  MacMahon's series.   
 Therefore, we find that\footnote{ The deduction above also appears in the work by Andrews-Rose and Rose \cite{Andrews-Rose, Rose}.}

  \begin{align*}
  \sum_{n=-\infty}^{\infty}\frac{q^{\frac{n(n+1)}{2}}}{(q;q)_{\infty}^3}\cdot z^{2n}=\left(1+z^{-2}\right)\sum_{n=0}^{\infty}A_n(q)(z+z^{-1})^{2n}.
  \end{align*}
   
  Thanks to the Binomial Theorem, followed by a simple shift in the index of summation, and culminating with a change in the order of summation, we obtain
  \begin{align*}
   \sum_{n=-\infty}^{\infty}\frac{q^{\frac{n(n+1)}{2}}}{(q;q)_{\infty}^3}\cdot z^{2n}&=\left(1+z^{-2}\right)\sum_{n=0}^{\infty}A_n(q)\sum_{j=0}^{2n}\binom{2n}{j}z^{2j-2n}
   =\left(1+z^{-2}\right)\sum_{n=0}^{\infty}A_n(q)\sum_{j=-n}^{n}\binom{2n}{j+n}z^{2j}\\
   &=\left(1+z^{-2}\right)\sum_{j=-\infty}^{\infty}\sum_{n=|j|}^{\infty}\binom{2n}{j+n}A_n(q)z^{2j}.
  \end{align*}
  After multiplying through $(1+z^{-2})$, we obtain
  $$
   \sum_{n=-\infty}^{\infty}\frac{q^{\frac{n(n+1)}{2}}}{(q;q)_{\infty}^3}\cdot z^{2n}=\sum_{j=-\infty}^{\infty}\left(\sum_{n=|j|}^{\infty}\binom{2n}{j+n}A_n(q)+\sum_{n=|j+1|}^{\infty}\binom{2n}{j+n+1}A_n(q)\right)z^{2j}.
  $$
  The theorem follows by comparing the coefficient of $z^{2k}$ on both sides after making use of the binomial coefficient identity
 $\binom{m}{r}+\binom{m}{r+1}=\binom{m+1}{r+1}.$ 
   \end{proof}

\begin{proof}[Proof of Corollary~\ref{Ak_Approximation}]
To prove the corollary, we truncate the infinite sums in Theorem \ref{Ak} after $j$ terms,  and we then apply
(\ref{AkShift}) and  \eqref{GF-3-color}.
\end{proof}

\subsection{MacMahon's $C_k(q)$}
Here we prove Theorem~\ref{Ck} and Corollary~\ref{Ck_Approximation}.

\begin{proof}[Proof of Theorem~\ref{Ck}]
We again use the Jacobi triple product identity (\ref{JTP}).
After factoring out $(q^2;q^2)_{\infty}$, and then letting $z\rightarrow z^2$, a simple reindex gives
$$
\sum_{n=-\infty}^{\infty}q^{n^{2}}z^{2n}=(q^2;q^2)_{\infty}\prod_{n=1}^{\infty}(1+z^{-2}q^{2n-1})(1+z^2q^{2n-1}).
$$
One easily checks that
$$
\sum_{n=-\infty}^{\infty}q^{n^{2}}z^{2n}=(q^2;q^2)_{\infty}\prod_{n=1}^{\infty}\left((1-q^{2n-1})^2+(z+z^{-1})^2q^{2n-1}\right).
$$
 After factoring out $(q;q^2)_{\infty}^2$ from the infinite product,  we obtain
$$
\sum_{n=-\infty}^{\infty}q^{n^{2}}z^{2n}=(q^2;q^2)_{\infty}(q;q^2)_{\infty}^2\prod_{n=1}^{\infty}\left(1+\frac{q^{2n-1}}{(1-q^{2n-1})^2}\cdot (z+z^{-1})^2\right).
$$
Thanks to definition (\ref{Mac2}), we find that the infinite product on the right, as a power series in $(z+z^{-1})^2,$ is the generating function for  MacMahon's series.   
Namely, we have
$$
\sum_{n=-\infty}^{\infty}\frac{q^{n^{2}}}{(q^2;q^2)_{\infty}(q;q^2)_{\infty}^2}\cdot z^{2n}=\sum_{n=0}^{\infty}C_n(q)(z+z^{-1})^{2n}.
$$
 Thanks to the Binomial Theorem, followed by a simple shift in the index of summation, and culminating with a change in the order of summation, we get
\begin{align*}
\sum_{n=-\infty}^{\infty}\frac{q^{n^{2}}}{(q^2;q^2)_{\infty}(q;q^2)_{\infty}^2}\cdot z^{2n}&=\sum_{n=0}^{\infty}C_n(q)\sum_{j=0}^{2n}\binom{2n}{j}z^{2j-2n}
=\sum_{n=0}^{\infty}C_n(q)\sum_{j=-n}^{n}\binom{2n}{j+n}z^{2j}\\
&=\sum_{j=-\infty}^{\infty}\sum_{n=|j|}^{\infty}\binom{2n}{j+n}C_n(q)z^{2j}
\end{align*}
By comparing the coefficient of $z^{2k}$ on both sides, one easily deduces claim (2), which in turn implies claim (1).
\end{proof}

\begin{proof}[Proof of Corollary~\ref{Ck_Approximation}]
By direct computation, for every positive integer $m$ we have
$$
q^{-k^2}C_{k+m}(q)=q^{m(m+2k)}+\dots.
$$
By truncating the infinite sums  in Theorem \ref{Ck} (2) after $j$ terms, the corollary now follows from this fact and (\ref{GenFcnOver}).

\end{proof}

\end{document}